\documentclass{JUSTC}

\usepackage{algorithm,algorithmicx,algpseudocode}%

\usepackage[utf8]{inputenc}
\usepackage{amsmath}
\usepackage{amsfonts}
\usepackage{amssymb}
\usepackage{amsthm}
\PassOptionsToPackage{normalem}{ulem}
\usepackage{ulem}
\usepackage{array}
\usepackage{cases}
\usepackage[numbers]{natbib}
\usepackage{dsfont}

\newcommand{\Tau}{\mathcal{T}}
\newcommand{\R}{\mathbb{R}}
\newcommand{\Px}{\mathbb{P}}
\newcommand{\Fx}{\mathbb{F}}
\newcommand{\Ex}{\mathbb{E}}
\newcommand{\F}{\mathcal{F}}
\newcommand{\N}{\mathbb{N}}

\newcommand{\1}{\mathds{1}}
\newcommand{\Afgh}{{\bf(A$_{fgh}$)}}
\newcommand{\Aghgrow}{{\bf(A$_{\overline{g}h}$)}}
\newcommand{\Afgoneside}{{\bf(A$_{one}$)}}
\newcommand{\Ahlam}{{\bf(A$_{h\lambda}$)}}
\newcommand{\Alya}{{\bf(A$_{Lya}$)}}
\newcommand{\Agrow}{{\bf(A$_{grow}$)}}
\newcommand{\Agrows}{{\bf(A$'_{grow}$)}}
\newcommand{\Ai}{{\bf(A$_{ini}$)}}
\newcommand{\Afg}{{\bf(A$_{\overline{f}\overline{g}h}$)}}
\newcommand{\Aab}{{\bf(A$_{\alpha\beta}$)}}
\newcommand{\Ahlr}{{\bf(A$_h$)}}
\definecolor{Red}{rgb}{1.00, 0.00, 0.00}

\definecolor{Blue}{rgb}{0.00, 0.00, 1.00}

\allowdisplaybreaks[4]

\type{Article}
\title{Mean field analysis of interacting network model with jumps}

\author[1,\Letter]{Zeqian Li}

\email{lzq7890@mail.ustc.edu.cn}

\runningtitle{Mean field analysis of interacting network model with jumps}
\runningauthor{Li}

\address{School of Mathematical Sciences, University of Science and Technology of China, Hefei, 230026, China.}

\abstract{
  This paper considers an $n$-particle jump-diffusion system with mean filed interaction, where the coefficients are locally Lipschitz continuous. We address the convergence as $n\to\infty$ of the empirical measure of the jump-diffusions to the solution of a deterministic McKean-Vlasov equation. The strong well-posedness of the associated McKean-Vlasov equation and a corresponding propagation of chaos result are proven. In particular, we provide also precise estimates of the convergence speed with respect to a Wasserstein-like metric.
}

\keywords{Mean-field theory; McKean-Vlasov equations; FPK equations; propagation of chaos; particle system}

\begin{document}

\maketitle 

\section{Introduction}

The interaction of particle systems is commonly observed in both physics and biological science. The behavior of such systems as the number of particles strength approaches infinity has sparked widespread interest, and the related study has been a long-standing issue. For instance, the interacting particles can stand for the stars in a galaxy\textsuperscript{\cite{dehnen2011n}}, the units of a neural network\textsuperscript{\cite{sirignano2020mean}} or the animals of a cluster\textsuperscript{\cite{bolley2011stochastic}}.

Mean field analysis is utilized to investigate systems that comprise a large number of interacting elements. When the scale of the particle system approaches infinity, the behavior of each element is assumed to be independent of others except for an average effect caused by the collective behavior of the system as a whole. Mean field analysis is widely employed in various fields, for example, describing nuclear structure and low-energy dynamics\textsuperscript{\cite{bender2003self}}, simulating bio-inspired neuronal networks\textsuperscript{\cite{touboul2014propagation, delarue2015particle}}, and modeling interbank lending and borrowing activities\textsuperscript{\cite{bo2015systemic}}. And its mathematical properties have been studied from various perspectives in recent years (see Refs.\cite{liu2023large, guillin2022uniform} for example). To describe such mean field phenomenon, the following mathematical framework is considered: the dynamics of the system with $n$ particles are governed by $n$ SDEs, and the mean field interactions can be expressed as a dependence on the empirical measure of the coefficients related to the equation. As $n$ tends to infinity, the dynamic of the limit system follows the McKean-Vlasov equation where the coefficients exhibit a natural dependence on the probability distribution associated with its solution. 

The convergence result for the empirical measure with respect to the $n$-particle system is called the propagation of chaos property. The concept of propagation of chaos was first proposed by Ref.\cite{kac1956foundations} to describe the asymptotic independence between particles that arises as the spatially extension tends to infinity. This convergence implies that the statistical properties of the system can be effectively described by the limiting measure, allowing for simplifications in the analysis and modeling of the system. For a more comprehensive overview, one can refer to Ref.\cite{sznitman1991topics}.

Since the seminal work Ref.\cite{mckean1966class}, there are various studies on the well-posedness of the McKean-Vlasov equations and the propagation of chaos for the mean field model without jumps, even under less strict hypothesis than Lipschitz continuity of the coefficients. For example, under relaxed regularity conditions, Ref.\cite{mishura2020existence} is a recent work that establishes weak and strong well-posedness results and Ref.\cite{daniel2018strong} considers also the propagation of chaos. Ref.\cite{liu2021long} investigates the uniform-in-time propagation of chaos for the mean-field weakly interacting particle system.

In this paper, we focus on the mean field coupled jump-diffusion particle system and its limit McKean-Vlasov equation. The jump term is modeled as a double stochastic Poisson process with the state-dependent intensity. The coefficients of corresponding McKean-Vlasov equation are locally Lipschitz continuous and satisfy a growth condition. We establish the strong well-posedness of the relevant McKean-Vlasov equations and prove a convergence result for the particle system as it extends spatially.

The first novelty of this paper is that we prove the strong well-posedness of the solution to the McKean-Vlasov equation under our framework. This result of the model with jumps has already been studied under globally Lipschitz assumptions on the coefficients\textsuperscript{\cite{andreis2018mckean}}. For the locally Lipschitz case, it becomes more challenging. Ref.\cite{erny2022well} uses Osgood’s lemma instead of Gr\"{o}nwall’s lemma to deal with the locally Lipschitz case, and Ref.\cite{mehri2020propagation} applies the Euler approximation to construct a solution of the McKean-Vlasov equation. In this paper, we first establish the existence of a local (strong) solution by a classical truncation argument, then we provide a criterion using Lyapunov functions and thereby prove that under the given growth conditions, such solution is a global one. As an example of our locally Lipschitz system, we show that a biological science model called FitzHugh-Nagumo neuron networks (which was first proposed by Ref.\cite{baladron2012mean}, with clarification note by Ref.\cite{bossy2015clarification}) falls into our model.

The second novelty of our paper is that we establish the propagation of chaos property and estimate the order of the convergence rate of our model. Under continuous frameworks, Refs.\cite{liping2018uniqueness, bo2022probabilistic} study the convergence rate. Refs.\cite{erny2022well, mehri2020propagation} prove the propagation of chaos with respect to the jump model under locally Lipschitz assumption, but without a conclusion on the convergence rate. In this paper, to prove the smoothness of the propagator, we study the well-posedness of an integro-differential parabolic equation. With the help of this smoothness property, we ultimately provide an estimate for the order of convergence rate with respect to a Wasserstein-like distance.

The paper is organized as follows: we introduce in Section \ref{sec:Model} the $n$-particle system with mean filed interaction, whose dynamic follows a couple of SDEs with jumps. In Section \ref{sec:MFa} we study the existence of a solution to the corresponding McKean-Vlasov limit equation. Section \ref{sec:Propagation} establishes the propagation of chaos under a proper metric, gives an estimate of the speed of the convergence and proves the uniqueness of the solution to the McKean-Vlasov limit equation.

\section{Model}\label{sec:Model}
\subsection{Jump diffusion model}

Let $(\Omega,\F,\Fx,\Px)$ be a complete filtered probability space with the filtration $\Fx=(\F_t)_{t\in[0,T]}$ satisfying the usual conditions. We consider an interacting system involving $n$ SDEs which is described as follows: for $i=1,\ldots,n$,
\begin{align}\label{eq:Intensityi}
    dX_t^i=f(t,X_t)dt+g(t,X_t)dW_t^i+h(t,X_{t-}^i)dN_t^{i},\quad X_0^i\in\R^m,
\end{align}
where $X_t=(X_t^1,\ldots,X_t^n)$ for $t\in[0,T]$ and $W^i=(W_t^i)_{t\in[0,T]}$ for $i=1,\ldots,n$ are $n$-independent ($d$-dimensional) Brownian motions, and for each $i=1,\ldots,n$, $N^{i}=(N_t^{i})_{t\in[0,T]}$ is a double stochastic Poisson processes with the state-dependent intensity given by the positive functions $\lambda_i(t,X_{t-})$ for $t\in[0,T]$. In other words, it holds that
\begin{align}\label{eq:pure-jumpmartingale}
    M_t^{i}:= N_t^{i} - \int_0^t\lambda_i(s,X_{s-})ds,\quad t\in[0,T]
\end{align}
is a $(\Px,\Fx)$-martingale.

We impose the following assumption on the coefficients $(f,g,h)$ of the system \eqref{eq:Intensityi}, so that this system admits a unique strong solution $X=(X_t)_{t\in[0,T]}$ which takes value on $\R^{mn}$.

\begin{itemize}
\item[\Afgh] The coefficients $f:[0,T]\times\R^{mn}\to\R^m$, $g:[0,T]\times\R^{mn}\to\R^{md}$ and $h:[0,T]\times\R^m\to\R^m$ satisfy the uniformly locally Lipschitz condition. Namely, for any $R>0$, there exists a constant $L_R>0$ such that for all $t\in[0,T]$ and $x=(x_1,\ldots,x_n)$, $y=(y_1,\ldots,y_n)\in B_R(0):=\{x\in\R^{mn};\ |x|\leq R\}$,
\begin{align*}
    |f(t,x)-f(t,y)| + |g(t,x)-g(t,y)| + |h(t,x_i)-h(t,y_i)|\leq L_R|x-y|,\\
    \forall i=1,\ldots,n.
\end{align*}
where $|\cdot|$ denotes the Euclidean norm.
\end{itemize}

Under a truncation argument, the condition {\Afgh} can only guarantee the existence and uniqueness of local (strong) solutions of \eqref{eq:Intensityi}. In order to establish the global solution, we further impose the condition on the existence of a Lyapunov function associated to the system \eqref{eq:Intensityi} and give the following lemma:
\begin{lemma}\label{lem:solsystem}
Let the assumption {\Afgh} holds, also suppose the following:
\begin{itemize}
\item[\Alya] There exists a function $V:\R^{mn}\to\R_+:=(0,+\infty)$ satisfying that
\begin{enumerate}[label=(\roman*)]
\item Let $q_R:=\inf_{|x|>R}V(x)$, then $\lim_{R\to\infty}q_R=+\infty$;
\item $\Ex[V(X_0)]<\infty$;
\item There exists a constant $C>0$, such that for all $t\in[0,T]$,
\begin{align*}
    \Ex[V(X_{t\wedge\tau_R})]\leq\Ex[V(X_0)]+C\int_0^t\left(1+\Ex[V(X_{s\wedge\tau_R})]\right)ds,
\end{align*}
where $\tau_R:=\inf\{t\in[0,T);\ |X_t|\geq R\}$, $\inf\emptyset:=T$.
\end{enumerate}
\end{itemize}
Then there exists a unique strong solution $X=(X_t)_{t\in[0,T]}$ of the system \eqref{eq:Intensityi} with an initial condition satisfying $\Ex[|X_0|^2]<\infty$.
\end{lemma}

\begin{proof}
By applying Theorem 9.1 of chapter IV of Ref.\cite{ikeda2014stochastic}, the condition {\Afgh} guarantees the existence and uniqueness of local (strong) solutions to \eqref{eq:Intensityi}. In other words, we have the solution of \eqref{eq:Intensityi} on the time interval $[0,t\wedge\tau_R]$ for all $t\in[0,T]$ and $R>0$. Thus, it is enough to prove that $\tau_R\to T$ as $R\to\infty$, $\Px$-a.s. Note that the condition (iii) of {\Alya} yields that for all $t\in[0,T]$ and $R>0$, $\Ex[V(X_{t\wedge\tau_R})]\leq e^{Ct}\{1+\Ex[V(X_0)]\}$. Therefore, we have for all $t\in[0,T]$,
\begin{align}
&\Px(\tau_R<t) \leq \frac{1}{q_R}\Ex\left[\1_{\tau_R<t}V(X_{\tau_R})\right]\nonumber\\
=&\frac{1}{q_R}\Ex\left[\1_{\tau_R<t}V(X_{t\wedge\tau_R})\right]\leq\frac{1}{q_R} {\rm e}^{Ct}\left(1+\Ex[V(X_0)]\right),
\end{align}
and hence $\Px(\tau_R<t)\to0$ as $R\to\infty$ by the condition (i) of {\Alya}, which gives the desired result.
\end{proof}

We next give a condition on the growth of the coefficients of the system \eqref{eq:Intensityi}:
\begin{itemize}
\item[\Agrow] The coefficients $f:[0,T]\times\R^{mn}\to\R^m$, $g:[0,T]\times\R^{mn}\to\R^{md}$, $h:[0,T]\times\R^m\to\R^m$ and $\lambda_i:[0,T]\times\R^{mn}\to\R_+$ satisfy the following growth condition: there exists a constant $K>0$ such that for all $t\in[0,T]$ and $x=(x_1,\ldots,x_n)\in\R^{mn}$,
\begin{align*}
    2\sum_{i=1}^n x_i^{\top}f(t,x)+{\rm tr}\left[gg^{\top}(t,x)\right]+\sum_{i=1}^n\lambda_i(t,x)\left(|h(t,x_i)|^2+2x_i^{\top}h(t,x_i)\right)\\
    \leq K(1+|x|^2).
\end{align*}
\end{itemize}

It is not difficult to verify that if $\Ex[|X_0|^2]<\infty$, then $V(x)=|x|^2$ for $x\in\R^{mn}$ satisfies the conditions (i)-(iii) of {\Alya} under the assumption {\Agrow}. Then, Lemma \ref{lem:solsystem} yields that:
\begin{lemma}\label{lem:solsystem2}
Let assumptions {\Afgh} and {\Agrow} hold. Then there exists a unique strong solution $X=(X_t)_{t\in[0,T]}$ of the system \eqref{eq:Intensityi} with an initial condition satisfying $\Ex[|X_0|^2]<\infty$.
\end{lemma}

\subsection{Mean field model}

The particle system considered in this paper is given by the following mean field model:
\begin{align}\label{eq:Intensityimfm}
    dX_t^i&=\overline{f}\left(t,X_t^i,\frac{1}{n}\sum_{j=1}^{n}\alpha(X_{t}^{i},X_{t}^{j})\right)dt+\overline{g}\left(t,X_t^i,\frac{1}{n}\sum_{j=1}^{n}\beta(X_{t}^{i},X_{t}^{j})\right)dW_t^i\nonumber\\
    &\quad+h(t,X_{t-}^i)dN_t^i,
\end{align}
for $i=1,\ldots,n$, where $\alpha:\R^m\times\R^m\to\R$, $\beta:\R^m\times\R^m\to\R$, $\overline{f}:[0,T]\times\R^m\times \R\to\R^m$, $\overline{g}:[0,T]\times\R^m\times\R\to\R^{md}$ and $h:[0,T]\times\R^m\to\R^m$ are all measurable functions.

Let us define the following functions: for $x=(x^1,\ldots,x^n)\in\R^{mn}$,
\begin{align}\label{eq:fg}
    f_i(t,x)&:= \overline{f}\left(t,x^i,\frac{1}{n}\sum_{j=1}^{n}\alpha(x^{i},x^{j})\right),\quad g_i(t,x):= \overline{g}\left(t,x^i,\frac{1}{n}\sum_{j=1}^{n}\beta(x^{i},x^{j})\right).
\end{align}
If $(f_i,g_i,h)$ above satisfies assumptions {\Afgh} and {\Agrow}, then by Lemma \ref{lem:solsystem2} the system \eqref{eq:Intensityimfm} admits a unique strong solution $X=(X^1,\ldots,X^n)$ provided that $\Ex[|X_0|^2]<\infty$.

\begin{example}
Ref.\cite{baladron2012mean} address the model of the network of FitzHugh-Nagumo neurons with simple maximum conductance variation, and such model is also investigated by Refs.\cite{Eric2014Meanfield, ermentrout2010mathematical, hocquet2021optimal}. However, the aforementioned studies predominantly concentrate on research conducted within the framework of diffusion state equation models. We next provide a verification that the network model with one population in Ref.\cite{baladron2012mean} falls into our jump-diffusion system \eqref{eq:Intensityimfm}. To this purpose, let $m=3$ and $d=3$. For $i=1,\ldots,n$ and $x^i=(x^i_1,x_2^i,x_3^i)\in\R\times\R\times[0,1]$, define the coefficients as $\alpha(x^i,x^j):=-J(x_1^{i}-V_{rev})x_3^{j}$ and $\beta(x^i,x^j):=-\sigma(x_1^{i}-V_{rev})x_3^{j}$, where $J,V_{rev},\sigma\in\R_+$. For any $(t,x^i,z)\in[0,T]\times\R^3\times\R$, we introduce the following functions:
\begin{align}
\overline{f}(t,x^{i},z)&=\begin{bmatrix}
x_1^{i}-\frac{(x_1^{i})^{3}}{3}-x_2^{i}+I_t+z \\
c(x_1^{i}+a-bx_2^{i})\\
\frac{a_{r}T}{1+e^{-\theta(x_1^{i}-V)}}(1-x_3^{i})-a_{d}x_3^{i}
\end{bmatrix},\\
\overline{g}(t,x^{i},z)&=\begin{bmatrix}
\sigma_{ext} & 0 & z\\
0      & 0 & 0\\
0      & 0 & \chi(x_3^{i})\sqrt{\frac{a_{r}T}{1+e^{-\theta(x_1^{i}-V)}}(1-x_3^{i})+a_{d}x_3^{i}}
\end{bmatrix},
\end{align}
where $\chi(x_3^{i})=0.1{\rm e}^{-\frac{1}{2(1-(2x_3^{i}-1)^{2})}}$ and $I:[0,T]\to\R$ is a deterministic bounded function. In addition, $V$, $a,b,c$, $a_r,a_d$ and $\sigma_{ext}$ are all positive constants.

Recall the functions $f,g$ defined by \eqref{eq:fg} with $\overline{f},\overline{g}$ and $\alpha,\beta$ given as above. Note that $\chi:\R\to\R_+$ is Lipschitz continuous. Then, it is not difficult to verify that $f,g$ satisfy assumptions {\Afgh} and {\Agrow} with $h\equiv0$. Let us further impose the following condition:
\begin{itemize}
\item[{\Ahlam}] The intensity function $\lambda_i:[0,T]\times\R^{mn}\to\R^m$ of the $i$-th state and the jump function $h:[0,T]\times\R^m\to\R^m$ satisfy that:
\begin{enumerate}[label=(\roman*)]
    \item For any $R>0$, there exists a constant $L_R>0$ which is independent of $t$, such that for all $t\in[0,T]$ and $x=(x^1,\ldots,x^n),y=(y^1,\ldots,y^n)\in B_R(0):=\{x\in\R^{mn};\ |x|\leq R\}$ (where $x^i,y^i\in\R^m$),
    \begin{align*}
        |h(t,x^i)-h(t,y^i)|\leq L_R|x^i-y^i|, \quad \forall i=1,\ldots,n;
    \end{align*}
    \item For $i=1,\ldots,n$, $\lambda_i(t,x)=p_i\psi(t,\overline{x})$ where $\psi:[0,T]\times\R^m\to\R_+$ is measurable and $\overline{x}:=\frac{1}{n}\sum_{i=1}^n x^i$.
    \item There exists a constant $K>0$ such that for all $t\in[0,T]$ and $x=(x^1,\ldots,x^n)\in\R^{mn}$ with $x^i\in\R^m$,
\end{enumerate}
\begin{align*}
    \sum_{i=1}^n p_i \psi(t,\overline{x})(|h(t,x^i)|^2+2x^{i \top}h(t,x^i))\leq K(1+|x|^2).
\end{align*}
\end{itemize}
Therefore, we have established a jump-diffusion model of the network of FitzHugh-Nagumo neurons with simple maximum conductance variation as an extension to the one in Ref.\cite{baladron2012mean}, where the jump process models an intense response of the neuron when it faces an external environmental impact, allowing for a more realistic representation of spiking activity and the integration of multiple inputs within a neuron.
\end{example}
\section{Mean field analysis}\label{sec:MFa}

In the following sections, we only consider the case that the intensity function of the $i$-th state is given by the form $\lambda_i(t,x)=p_i\psi(t,\overline{x})$ where $p_i\in\R_+$ and $\psi:[0,T]\times\R^m\to\R_+$ is a bounded measurable function with upper bound $C_{\psi}$. We introduce the empirical measure-value process as
\begin{align}\label{eq:valueproc}
    \mu_t^n(dp,dx):=\frac{1}{n}\sum_{i=1}^n\delta_{(p_i,X_t^i)}(dp,dx),
\end{align}
where $\delta$ denotes the Dirac delta measure, $(p,x)\in E:=\R_+\times\R^m$ and $\mu_0^n(dp,dx_0)$ is the empirical measure related to the initial state. We propose the following assumption:
\begin{itemize}
\item[{\Ai}] There exists $p\in\R$, such that $p_i\to p$ as $i\to\infty$. The sequence of r.v.s $(X_0^i)_{i\ge 1}$ is i.i.d. and satisfies that $\Ex[|X_0^i|^2]<\infty$ for all $i\geq1$.
\end{itemize}

Let ${\cal P}(E)$ be the set of all Borel probability measures on $E$. Under assumption {\Ai}, we know that there exists a probability measure $\rho_0\in {\cal P}(E)$ such that $\mu_0^n\Rightarrow\rho_0$ in the weak sense.
\\

Assume $\phi\in C^2(E)$ is a real-valued twice continuously differentiable function, let $\partial_{x}\phi:=(\partial_{x_{i}}\phi)_{i=1,\ldots,n}^{\top}$ be the gradient of $\phi$ and $\partial_{x}^2\phi$ be the corresponding Hessian matrix. By \eqref{eq:valueproc}, it holds for $s\in[0,T]$ that
\begin{align}
\left<\mu_s^n,\alpha(x,\cdot)\right>=\frac{1}{n}\sum_{j=1}^{n}\alpha(x,X_{s}^{j}),\quad
\left<\mu_s^n,\beta(x,\cdot)\right>=\frac{1}{n}\sum_{j=1}^{n}\beta(x,X_{s}^{j}).
\end{align}
Then, for $i=1,\ldots,n$ and $t\in [0,T]$, It\^o's formula yields that
\begin{align}
    \phi(p_i,X_t^i)=&\phi(p_i,X_0^i)+ \int_0^t \partial_{x}\phi(p_i,X_s^i)^{\top}\overline{f}(s,X_s^i,\left<\mu_s^n,\alpha(X_s^i,\cdot)\right>)ds\nonumber\\
    &+\frac{1}{2}\int_0^t{\rm tr}\left[\partial_{x}^2\phi(p_i,X_s^i)\overline{g}\overline{g}^{\top}\left(s,X_s^i,\left<\mu_s^n,\beta(X_s^i,\cdot)\right>\right)\right]ds\nonumber\\
    &+\int_0^t p_i\left(\phi\left(p_i,X_{s}^i+h(s,X_{s}^i)\right)-\phi\left(p_i,X_{s}^i\right)\right)\psi\left(s,\overline{X}_s\right)ds\nonumber\\
    &+\int_0^t \partial_{x}\phi\left(p_i,X_s^i\right)^{\top}\overline{g}\left(s,X_s^i,\left<\mu_s^n,\beta(X_s^i,\cdot)\right>\right)dW_s^i\nonumber\\
    &+\int_0^t\left(\phi\left(p_i,X_{s-}^i+h(s,X_{s-}^i)\right)-\phi\left(p_i,X_{s-}^i\right)\right)dM_s^i.
\end{align}
Define the following operators as: for $(t,p,x)\in[0,T]\times E$ and $\nu\in{\cal P}(E)$,
\begin{align}\label{eq:operators}
\begin{cases}
    \displaystyle {\cal L}_c^{t,\nu}\phi(p,x):=\partial_{x}\phi(p,x)^{\top}\overline{f}(t,x,\left<\nu,\alpha(x,\cdot)\right>)+\frac{1}{2}{\rm tr}\left[\partial_{x}^2\phi(p,x)\overline{g}\overline{g}^{\top}(t,x,\left<\nu,\beta(x,\cdot)\right>)\right],\\ \\
    \displaystyle {\cal L}_j^{t,\nu}\phi(p,x):=p\left(\phi(p,x+h(t,x))-\phi(p,x)\right)\psi(t,\left<\nu,I\right>),
\end{cases}
\end{align}
where $I(p,x)=x$. Therefore, it holds that
\begin{align}\label{fpkn}
    \left<\mu_t^n,\phi\right>=&\left<\mu_0^n,\phi\right>+\int_0^t\left<\mu_s^n,{\cal L}_c^{s,\mu_s^n}\phi+{\cal L}_j^{s,\mu_s^n}\phi\right>ds\nonumber\\
    &+\frac{1}{n}\sum_{i=1}^n\int_0^t \partial_{x}\phi(p_i,X_s^i)\overline{g}(s,X_s^i,\left<\mu_s^n,\beta(X_s^i,\cdot)\right>)dW_s^i\nonumber\\
    &+\frac{1}{n}\sum_{i=1}^n\int_0^t\left(\phi(p_i,X_{s-}^i+h(s,X_{s-}^i))-\phi(p_i,X_{s-}^i)\right)dM_s^i.
\end{align}
Taking formally $n\to\infty$ on both side of \eqref{fpkn} shows that any limit point of $\mu^n$ should satisfy the following FPK equation:
\begin{align}\label{eq:limit-mu}
    \left<\mu_t,\phi\right>&=\left<\mu_0,\phi\right>+\int_0^t\left<\mu_s,{\cal L}_c^{s,\mu_s}\phi+{\cal L}_j^{s,\mu_s}\phi\right>ds.
\end{align}
We next establish a solution of the above FPK equation. Prior to do this, we introduce the following state equation given by
\begin{align}\label{eq:limit-sol-FPK}
    \begin{cases}
        \displaystyle \mu_t(dp,dx) := \int_{E}\Ex\left[\delta_{X_t^{p,x_0,\mu}}(dx)\right]\rho_0(dp,dx_0),\quad {\rm on}~{\cal B}(E),\\ \\
        \displaystyle X_t^{p,x_0,\mu} = x_0+\int_0^t \overline{f}(s,X_s^{p,x_0,\mu},\left<\mu_s,\alpha(X_s^{p,x_0,\mu},\cdot)\right>)ds \\
        \displaystyle \qquad\qquad\quad+ \int_0^t \overline{g}(s,X_s^{p,x_0,\mu},\left<\mu_s,\beta(X_s^{p,x_0,\mu},\cdot)\right>)dW_s+\int_0^t h(s,X_{s-}^{p,x_0,\mu})dN_s,
    \end{cases}
\end{align}
where $N=(N_t)_{t\in[0,T]}$ is a double stochastic Poisson process with intensity process $p\psi(t, \left<\mu_t,I\right>)$, namely, the stochastic process $M_t:=N_t-\int_0^tp\psi(t,\left<\mu_s,I\right>)ds$ for $t\in[0,T]$ is a martingale.

For $(s,y)\in[0,T]\times\R^m$, define 
\begin{align}
    \left\{
    \begin{array}{lcr}
        \tilde{\alpha}(s,y):=\int_{E}\Ex[\alpha(y,X_s^{p,x_0,\mu})]\rho_0(dp,dx_0)\left(=\left<\mu_s,\alpha(y,\cdot)\right>\right),\\ \\
        \tilde{\beta}(s,y):=\int_{E}\Ex[\beta(y,X_s^{p,x_0,\mu})]\rho_0(dp,dx_0)\left(=\left<\mu_s,\beta(y,\cdot)\right>\right),
    \end{array}
    \right.
\end{align}
then the equation \eqref{eq:limit-sol-FPK} can be rewritten as
\begin{align}\label{eq:Z-FPK}
    \begin{cases}
        \displaystyle X_0^{p,x_0,\mu}=x_0,\\ \\
        \displaystyle dX_t^{p,x_0,\mu}= \overline{f}(t,X_t^{p,x_0,\mu},\tilde{\alpha}(t,X_t^{p,x_0,\mu}))dt + \overline{g}(t,X_t^{p,x_0},\tilde{\beta}(t,X_t^{p,x_0,\mu}))dW_t\\
        \displaystyle\qquad\qquad\quad+ h(t,X_{t-}^{p,x_0,\mu})dN_t.
    \end{cases}
\end{align}
where the corresponding intensity of $N_t$ is $p\psi(t, \Ex[X_t^{p,x_0,\mu}])$.

To establish the well-posedness of equation \eqref{eq:Z-FPK}, we give the corresponding assumptions:
\begin{itemize}
\item[\Afg] The coefficients $\overline{f}:[0,T]\times\R^{m}\times\R\to\R^m$, $\overline{g}:[0,T]\times\R^{m}\times\R\to\R^{md}$ and $h:[0,T]\times\R^m\to\R^m$ satisfy the uniformly locally Lipschitz condition. Namely, for any $R>0$, there exists a constant $\overline{L}_R>0$ such that for all $t\in[0,T]$, $x_1,x_2\in B_R(0):=\{x\in\R^{m};\ |x|\leq R\}$ and $y_1,y_2\in \R$,
\begin{align*}
    |\overline{f}(t,x_1,y_1)-\overline{f}(t,x_2,y_2)| + |\overline{g}(t,x_1,y_1)-\overline{g}(t,x_2,y_2)| \\
    + |h(t,x_1)-h(t,x_2)|\leq \overline{L}_R(|x_1-x_2|+|y_1-y_2|).
\end{align*}
\end{itemize}
\begin{itemize}
\item[\Aab] The coefficients $\alpha:\R^m\times\R^m\to\R$ and $\beta:\R^m\times\R^m\to\R$ satisfy the uniformly global Lipschitz condition and a growth condition. Namely, there exists a constant $J>0$ such that for all $x_1,x_2,y_1,y_2\in\R^m$, 
\begin{gather*}
    |\alpha(x_1,y_1)-\alpha(x_2,y_2)|+|\beta(x_1,y_1)-\beta(x_2,y_2)|\leq J(|x_1-x_2|+|y_1-y_2|),\\
    |\alpha(x_1,y_1)|+|\beta(x_1,y_1)|\leq J(1+|x_1|+|y_1|).
\end{gather*}
\end{itemize}
\begin{itemize}
\item[\Agrows] The coefficients $\overline{f}:[0,T]\times\R^m\times\R\to\R^m$, $\overline{g}:[0,T]\times\R^m\times\R\to\R^{md}$ and $h:[0,T]\times\R^m\to\R^m$ satisfy the following growth condition: there exists a constant $\overline{K}>0$ such that for all $t\in[0,T]$ and $x,y\in\R^m$,
\begin{gather*}
     2x^{\top}\overline{f}(t,x,y)+{\rm tr}\left[\overline{g}\overline{g}^{\top}(t,x,y)\right]\leq \overline{K}(1+|x|^2+|y|^2),\\
     |h(t,x)|^2+2x^{\top}h(t,x)\leq\overline{K}(1+|x|^2).
\end{gather*}
\end{itemize}

\begin{lemma}\label{lem:solsystem3}
Let assumptions {\Afg}, {\Aab} and {\Agrows} hold. Then there exists a unique strong solution $X^{p,x_0}=(X^{p,x_0}_t)_{t\in[0,T]}$ of the equation \eqref{eq:Z-FPK} with an initial condition $(p,x_0)\in E$.
\end{lemma}

\begin{proof}
By combining the two assumptions {\Afg} and {\Aab}, we find that the coefficients of the drift term and diffusion term of \eqref{eq:Z-FPK} actually satisfy the same locally Lipschitz condition as in assumption {\Afgh}, hence \eqref{eq:Z-FPK} admits a unique local strong solution.

Subsequently, we show that this solution is actually a global one. Let $V(x)=|x|^2$, for all $x\in \R^m$. It is not hard to check $V(x)$ satisfies (i) and (ii) of assumption {\Alya}. For (iii), an application of It\^o's formula leads to
\begin{align}\label{eq:itoV}
    V(X_{t \wedge \tau_R}^{p,x_0,\mu})=&x_0^2+\int_0^{t\wedge \tau_R} 2(X_s^{p,x_0,\mu})^{\top}\overline{f}(s,X_{s}^{p,x_0,\mu},\tilde{\alpha}(s,X_s^{p,x_0,\mu}))ds\nonumber\\
    &+\int_0^{t\wedge \tau_R} 2(X_s^{p,x_0,\mu})^{\top}\overline{g}(s,X_s^{p,x_0,\mu},\tilde{\beta}(s,X_s^{p,x_0,\mu}))dW_s\nonumber\\
    &+\int_0^{t\wedge \tau_R}{\rm tr}\left[\overline{g}\overline{g}^{\top}(s,X_s^{p,x_0,\mu},\tilde{\beta}(s,X_s^{p,x_0,\mu}))\right]ds\nonumber\\
    &+\int_0^{t\wedge \tau_R}\left((X_{s }^{p,x_0,\mu}+h(s,X_{s}^{p,x_0,\mu}))^2-(X_{s }^{p,x_0,\mu})^{2}\right)dN_s.
\end{align}
Taking expectation on both sides of equation \eqref{eq:itoV}, by assumption {\Agrows} and {\Aab} it holds that for some positive constant $C=C_{\overline{f},\overline{g},h,p,\psi,\alpha,\beta}>0$,
\begin{align}
    &\Ex[V(X_{t \wedge \tau_R}^{p,x_0,\mu})]\nonumber\\
    =&x_0^2+\Ex\left[\int_0^{t\wedge \tau_R} 2(X_s^{p,x_0,\mu})^{\top}\overline{f}\left(s,X_s^{p,x_0,\mu},\left.\Ex[\alpha(y,X_s^{p,x_0,\mu})]\right|_{y=X_s^{p,x_0,\mu}}\right)ds\right]\nonumber\\
    &+\Ex\left[\int_0^{t\wedge\tau_R}\overline{g}\overline{g}^{\top}\left(s,X_s^{p,x_0,\mu},\left.\Ex[\beta(y,X_s^{p,x_0,\mu})]\right|_{y=X_s^{p,x_0,\mu}}\right)ds\right]\nonumber\\
    &+\Ex\left[\int_0^{t\wedge\tau_R} p\psi(s,\Ex[X_s^{p,x_0,\mu}])\left(|h(s,X_{s }^{p,x_0,\mu})|^2+2(X_s^{p,x_0,\mu})^{\top}h(s,X_s^{p,x_0,\mu})\right)ds\right]\nonumber\\
    \leq&\Ex[(X_0^{p,x_0,\mu})^{2}]+C\Ex\bigg[\int_0^{t\wedge\tau_R}\Big(1+|X_s^{p,x_0,\mu}|^2+\left|\left.\Ex[\alpha(y,X_s^{p,x_0,\mu})]\right|_{y=X_s^{p,x_0,\mu}}\right|^2\nonumber\\
    &+\left|\left.\Ex[\beta(y,X_s^{p,x_0,\mu})]\right|_{y=X_s^{p,x_0,\mu}}\right|^2\Big)ds\bigg]\nonumber\\
    \leq&\Ex[(X_0^{p,x_0,\mu})^{2}]+C\Ex\bigg[\int_0^{t\wedge\tau_R}\Big(1+|X_s^{p,x_0,\mu}|^2+\left|\Ex[\left|X_s^{p,x_0,\mu}\right|]\right|^2\Big)ds\bigg]\nonumber\\
    \leq& \Ex[(X_0^{p,x_0,\mu})^{2}]+C\int_0^t\left(\Ex[|X_{s\wedge\tau_R}^{p,x_0,\mu}|^2]+1\right)ds,
\end{align}
which implies that $V(x)$ also satisfies (iii). Then some similar estimates as in the proof of Lemma \ref{lem:solsystem} yield that the equation \eqref{eq:Z-FPK} with an initial condition $(p,x_0)\in E$ has a unique (global) strong solution.
\end{proof}

By Lemma \ref{lem:solsystem3}, the system \eqref{eq:Z-FPK} admits a unique strong solution. Equivalently, we have proved the strong existence and uniqueness of $\mu=(\mu_t)_{t\in[0,T]}$ defined in the system \eqref{eq:limit-sol-FPK}. Now we show that this measure is indeed a solution of the corresponding FPK equation.

\begin{theorem}\label{lem:soldetermindMVE}
Let the assumptions {\Ai}, {\Afg}, {\Aab} and {\Agrows} hold. Then $\mu=(\mu_t)_{t\in[0,T]}$ defined in \eqref{eq:limit-sol-FPK} is a solution of the FPK equation \eqref{eq:limit-mu}.
\end{theorem}

\begin{proof}
For any $\phi\in C^2(E)$, by equation \eqref{eq:limit-sol-FPK} and It\^o's formula, it holds that
\begin{align}
    \phi(p,X_t^{p,x_0})=&\phi(p,X_0^{p,x_0})+ \int_0^t {\cal L}_c^{s,\mu_s}\phi(p,X_s^{p,x_0})ds\nonumber\\
    &+\frac{1}{2}\int_0^t \partial_{x}\phi(p,X_s^{p,x_0})\overline{g}(s,X_s^{p,x_0},\left<\mu_s,\beta(X_s^{p,x_0},\cdot)\right>)dW_s\nonumber\\
    &+\int_0^t \left(\phi(p,X_{s-}^{p,x_0}+h(s,X_{s-}^{p,x_0}))-\phi(p,X_{s-}^{p,x_0})\right)dN_s.
\end{align}
Taking expectations on both sides of this equation to kill the martingale term and integrating with respect to $\rho_0$ on $E$, we have
\begin{align}
    &\int_E\Ex[\phi(p,X_t^{p,x_0})]\rho_0(dp,dx_0)\nonumber\\
    =&\int_E\Ex[\phi(p,X_0^{p,x_0})]\rho_0(dp,dx_0)+\int_0^t\int_E \Ex[({\cal L}_c^{s,\mu_s}+{\cal L}_j^{s,\mu_s})\phi(p,X_s^{p,x_0})\rho_0(dp,dx_0)]ds\nonumber\\
    =&\int_E\phi(p,x_0)\rho_0(dp,dx_0)+\int_0^t\int_E\Ex[({\cal L}_c^{s,\mu_s}+{\cal L}_j^{s,\mu_s})\phi(p,X_s^{p,x_0})\rho_0(dp,dx_0)]ds.
\end{align}
This shows that $\mu$ satisfies the FPK equation \eqref{eq:limit-mu}, that is 
\begin{align}
\left<\mu_t,\phi\right>&=\left<\mu_0,\phi\right>+\int_0^t\left<\mu_s,{\cal L}_c^{s,\mu_s}\phi+{\cal L}_j^{s,\mu_s}\phi\right>ds.
\end{align}
Thus, we complete the proof.
\end{proof}

\section{Propagation of chaos}\label{sec:Propagation}

From Theorem \ref{lem:soldetermindMVE} we know that there exists at least one solution to the FPK equation \eqref{eq:limit-mu}. We fix once and for all one such solution $\mu_t$. Next we establish the propagation of chaos of the FPK equation \eqref{eq:limit-mu} (which means a convergence result of the empirical measure of the particle system to the solution of the FPK equation \eqref{eq:limit-mu} under a suitable distance) and thereby show that such solution of \eqref{eq:limit-mu} is actually unique.

\subsection{Propagator of FPK equation}

We first give the definition of propagator corresponding to \eqref{eq:limit-mu}.

\begin{definition}\label{def:propagator}
For $(p,x)\in E$ and $0\leq t\leq u\leq T$, the propagator corresponding to the probability solution $\mu=(\mu_t)_{t\in[0,T]}$ of \eqref{eq:limit-mu} is defined as
\begin{align}\label{propagator}
P_{t,u}\phi(p,x):=\Ex\left[\phi(p,X_u^{t,p,x,\mu})\right], 
\end{align}
where $\phi\in C_b^2(E)$ is bounded and twice continuously differentiable, and the process $X^{t,p,x,\mu}=(X_s^{t,p,x,\mu})_{s\in[t,T]}$ is the unique solution of the following SDE:
\begin{align}\label{eq:limit-sol-FPK2}
    X_s^{t,p,x,\mu}=&x+\int_t^s\overline{f}(r,X_r^{t,p,x,\mu},\left<\mu_r,\alpha(X_r^{t,p,x,\mu},\cdot)\right>)dr\nonumber\\
    &+\int_t^s\overline{g}(r,X_r^{t,p,x,\mu},\left<\mu_r,\beta(X_r^{t,p,x,\mu},\cdot)\right>)dW_r+\int_t^s h(r,X_{r-}^{t,p,x,\mu})dN_r,
\end{align}
where $N=(N_t)_{t\in[0,T]}$ is a double stochastic Poisson process with intensity process $p\psi(t,\left<\mu_t,I\right>)$, namely, the stochastic process $M_t:=N_t-\int_0^tp\psi(s,\left<\mu_s,I\right>)ds$ for $t\in[0,T]$ is a martingale.
\end{definition}

The following two lemmas give important properties of the propagator, which will be useful for estimating the convergence rate of $\mu^n$ to $\mu$.

\begin{lemma}\label{lem:propagator}
Let the assumptions {\Ai}, {\Afg}, {\Aab} and {\Agrows} hold, suppose further that $\overline{f}$, $\overline{g}$ is twice continuously differentiable, also suppose the following: 
\begin{itemize}
\item[\Ahlr] The coefficient $h:[0,T]\times\R^m\to\R^m$ is bounded, and there exist constants $L_h$, $L_h'>0$ such that for all $t\in[0,T]$ and $x_1$, $x_2\in\R^m$,
\begin{align*}
    L_h'|x_1-x_2|\leq|h(t,x_1)-h(t,x_2)|\leq L_h|x_1-x_2|.
\end{align*}
\end{itemize}
Then the propagator in Definition \ref{def:propagator} satisfies that for all $\phi\in C_b^2(E)$, $0\leq t\leq u\leq T$ and $(p,x)\in E$,
\begin{align}\label{eq:propagator}
\begin{cases}
   \displaystyle \partial_t P_{t,u}\phi(p,x)+\left({\cal L}_c^{t,\mu_t}+{\cal L}_j^{t,\mu_t}\right)P_{t,u}\phi(p,x)=0,\\[1em]
   \displaystyle P_{u,u}\phi(p,x)=\phi(p,x).
    \end{cases}
\end{align}
where ${\cal L}_c^{t,\mu_t}$ and ${\cal L}_j^{t,\mu_t}$ are defined as in \eqref{eq:operators}. Moreover, $P_{\cdot,u}\phi(p,\cdot)\in C^{1,2}([0,u]\times\R^m)$.
\end{lemma}

\begin{proof}
We first consider the following parabolic equation with respect to $\theta: [0,u]\times\R^m\to\R$ for fixed $p$:
\begin{align}\label{eq:integro}
    \begin{cases}
        \displaystyle \partial_t \theta(t,x)+\left({\cal L}_c^{t,\mu_t}+{\cal L}_j^{t,\mu_t}\right)\theta(t,x)=0, \\[1em]
        \displaystyle \theta(u,x)=\phi(p,x).
    \end{cases}
\end{align}
This is a type of second order integro-differential parabolic
equation with Cauchy boundary. By Theorem 3.1 of chapter II of Ref.\cite{garroni1992green}, we know that \eqref{eq:integro} has a unique classical solution $\theta(t,x)\in C^{1,2}([0,u]\times\R^m)$.

Next, for all $(t,x)\in[0,u]\times\R^m$, by applying It\^o's formula to $\theta(u,X_u^{t,p,x,\mu})$, we obtain
\begin{align}\label{eq:theta}
    \theta(t,x)=&\theta(u,X_u^{t,p,x,\mu})-\int_t^u\left(\partial_s+{\cal L}_c^{s,\mu_s}+{\cal L}_j^{s,\mu_s}\right)\theta(s,X_s^{t,p,x,\mu})ds\nonumber\\
    &-\int_t^u\partial_x \theta(s,X_s^{t,p,x,\mu})\overline{g}(s,X_s^{t,p,x,\mu},\left<\mu_s,\beta(X_s^{t,p,x,\mu},\cdot)\right>)dW_s\nonumber\\
    &-\int_t^u \left(\theta(s,X_{s-}^{t,p,x,\mu})+h(s,X_{s-}^{t,p,x,\mu}))-\theta(s,X_{s-}^{t,p,x,\mu})\right)dM_s. 
\end{align}
Taking expectations on both sides of above equation, from equation \eqref{eq:theta} we deduce that
\begin{align}
    \theta(t,x)&=\Ex\left[\theta(u,X_u^{t,p,x,\mu})-\int_t^u\left(\partial_s+{\cal L}_c^{s,\mu_s}+{\cal L}_j^{s,\mu_s}\right)\theta(s,X_s^{t,p,x,\mu})ds\right]\nonumber\\
    &=\Ex\left[\theta(u,X_u^{t,p,x,\mu})\right]\nonumber\\
    &=\Ex\left[\phi\left(p,X_u^{t,p,x,\mu}\right)\right]\nonumber\\
    &=P_{t,u}\phi(p,x).
\end{align}
 This yields that $P_{\cdot,u}\phi(p,\cdot)\in C^{1,2}([0,u]\times\R^m)$ satisfies the PDE \eqref{eq:propagator}. Thus, we complete the proof.
\end{proof}

\begin{lemma}\label{lem:par_t}
Let conditions of Lemma \ref{lem:propagator} hold. Then for all $0\leq t\leq u\leq T$,
\begin{align}
    \partial_t\left<\mu_t,P_{t,u}\phi\right>=0.
\end{align}
\end{lemma}

\begin{proof}
It follows from Theorem \ref{lem:soldetermindMVE}, Lemma \ref{lem:propagator}, equation \eqref{eq:limit-sol-FPK} and an application of It\^o's formula to $P_{t,u}\phi(p, X_t^{0,p,x_0,\mu})$ that
\begin{align}
    &\left<\mu_t,P_{t,u}\phi\right>=\int_{\R^m}\Ex\left[ P_{t,u}\phi(p, X_t^{0,p,x_0,\mu})\right] \rho_0(dp,dx_0)\nonumber\\
    =&\int_{\R^m}\Ex\Big[ P_{0,u}\phi(p,X_0^{0,p,x_0,\mu})+\int_0^t (\partial_s+{\cal L}_c^{s,\mu_s}+{\cal L}_j^{s,\mu_s})P_{s,u}\phi(p,X_s^{0,p,x_0,\mu})ds\nonumber\\
    &+\int_0^t\partial_x P_{s,u}\phi(p,X_s^{0,p,x_0,\mu})\overline{g}(s,X_s^{0,p,x,\mu},\left<\mu_s,\beta(X_s^{0,p,x_0,\mu},\cdot)\right>)dW_s\nonumber\\
    &+\int_0^t \left(P_{s,u}\phi(p,X_{s-}^{0,p,x_0,\mu}+h(s,X_{s-}^{0,p,x_0,\mu}))-P_{s,u}\phi(p,X_{s-}^{0,p,x_0,\mu})\right)dM_s \Big]\rho_0(dp,dx_0)\nonumber\\
    =&\int_{\R^m}\Ex\left[ P_{0,u}\phi(p,X_0^{0,p,x_0,\mu})\right] \rho_0(dp,dx_0).
\end{align}
 It appears that $\left<\mu_t,P_{t,u}\phi\right>$ is independent of $t$, which indicates the desired result.
\end{proof}

The propagator defined by \eqref{propagator} can help us establish the following relation satisfied by $\mu_t^n-\mu_t$ for any $t\in[0,T]$:

\begin{lemma}\label{lem:weakconverge}
Let conditions of Lemma \ref{lem:propagator} hold. Then for any $t\in[0,T]$, it holds that
\begin{align}\label{eq:weakconverge}
    &\left<\mu_t^n-\mu_t,\phi\right>\nonumber\\
    =&\left<\mu_0^n-\mu_0,P_{0,t}\phi\right>+\frac{1}{n}\sum_{i=1}^{n}\int_0^t\partial_x P_{s,t}\phi(p_i,X_s^i)\overline{g}(s,X_s^i,\frac{1}{n}\sum_{j=1}^{n}\beta(X_s^i,X_s^j))dW_s^i\nonumber\\
    &+\frac{1}{n}\sum_{i=1}^{n}\int_0^t \left(P_{s,t}\phi(p_i,X_{s-}^i+h(s,X_{s-}^i))-P_{s,t}\phi(p_i,X_{s-}^i)\right)dM_s^i.
\end{align} 
\end{lemma}

\begin{proof}
Recall the state process of our particle system $X^i=(X_t^i)_{t\in[0,T]}$ defined by \eqref{eq:Intensityimfm} for
$i\geq1$. By applying It\^o's formula to $P_{r,t}\phi(p_i,X_t^i)$ for $r\in[0,t]$, it follows from Lemma \ref{lem:propagator} that
\begin{align}
    &P_{r,t}\phi(p_i, X_t^i)\nonumber\\
    =&P_{0,t}\phi(p_i,X_0^i)+\int_0^r(\partial_s+{\cal L}_c^{s,\mu_s}+{\cal L}_j^{s,\mu_s})P_{s,t}\phi(p_i,X_s^i)ds\nonumber\\
    &+\int_0^r\partial_x P_{s,t}\phi(p_i,X_s^i)\overline{g}\left(s,X_s^i,\frac{1}{n}\sum_{j=1}^{n}\beta(X_s^i,X_s^j)\right)dW_s^i\nonumber\\
    &+\int_0^r \left(P_{s,t}\phi(p_i,X_{s-}^i+h(s,X_{s-}^i))-P_{s,t}\phi(p_i,X_{s-}^i)\right)dM_s^i\nonumber\\
    =&P_{0,t}\phi(p_i,X_0^i)+\int_0^r\partial_x P_{s,t}\phi(p_i,X_s^i)\overline{g}\left(s,X_s^i,\frac{1}{n}\sum_{j=1}^{n}\beta(X_s^i,X_s^j)\right)dW_s^i\nonumber\\
    &+\int_0^r \left(P_{s,t}\phi(p_i,X_{s-}^i+h(s,X_{s-}^i))-P_{s,t}\phi(p_i,X_{s-}^i)\right)dM_s^i,
\end{align}
then in view of the definition of $\mu^n$, we have
\begin{align}\label{eq:weakconverge1}
    &\left<\mu_r^n,P_{r,t}\phi\right>=\frac{1}{n}\sum_{i=1}^{n}P_{r,t}\phi(p_i,X_{r}^i)\nonumber\\
    =&\left<\mu_0^n,P_{0,t}\phi\right>+\frac{1}{n}\sum_{i=1}^{n}\int_0^r\partial_x P_{s,t}\phi\left(p_i,X_s^i\right)\overline{g}\left(s,X_s^i,\frac{1}{n}\sum_{j=1}^{n}\beta(X_s^i,X_s^j)\right)dW_s^i\nonumber\\
    &+\frac{1}{n}\sum_{i=1}^{n}\int_0^r \left(P_{s,t}\phi(p_i,X_{s-}^i+h(s,X_{s-}^i))-P_{s,T}\phi(p_i,X_{s-}^i)\right)dM_s^i.
\end{align}
Note that Lemma \ref{lem:par_t} implies 
\begin{align}\label{eq:weakconverge2}
\left<\mu_r,P_{r,t}\phi\right>=\left<\mu_0,P_{0,t}\phi\right>, \ for\ all\ r\in[0,t].
\end{align} 
From \eqref{eq:weakconverge1} and \eqref{eq:weakconverge2}, we deduce that
\begin{align}\label{eq:weakconvergepf}
    &\left<\mu_r^n-\mu_r,P_{r,t}\phi\right>\nonumber\\
    =&\left<\mu_0^n-\mu_0,P_{0,t}\phi\right>+\frac{1}{n}\sum_{i=1}^{n}\int_0^r\partial_x P_{s,t}\phi(p_i,X_s^i)\overline{g}\left(s,X_s^i,\frac{1}{n}\sum_{j=1}^{n}\beta(X_s^i,X_s^j)\right)dW_s^i\nonumber\\
    &+\frac{1}{n}\sum_{i=1}^{n}\int_0^r \left(P_{s,T}\phi(p_i,X_{s-}^i+h(s,X_{s-}^i))-P_{s,T}\phi(p_i,X_{s-}^i)\right)dM_s^i.
\end{align}
In \eqref{eq:weakconvergepf}, let $r=t$ and note that $P_{t,t}\phi=\phi$ by virtue of Lemma \ref{lem:propagator}, then we get the desired result.
\end{proof}

Next we give the estimate of the gradient $\partial_x P_{s,t}\phi(p,x)$. To this purpose, we impose the following one-side Lipschitz condition:
\begin{itemize}
\item[\Afgoneside] The coefficient $\overline{f}:[0,T]\times\R^{m}\times \R \to\R^m$ satisfies a one-side Lipschitz condition. Namely, there exists a constant $L_{\overline{f}}$ such that for all $t\in[0,T]$, $x_1,x_2\in \R^{m}$ and $y_1,y_2\in \R$,
\begin{align*}
    \left<x_1-x_2, \ \overline{f}(t,x_1,y_1)-\overline{f}(t,x_2,y_2)\right>\leq L_{\overline{f}}(|x_1-x_2|^2+|y_1-y_2|^2).
\end{align*}
\end{itemize}

\begin{lemma}\label{lem:partialx}
Let conditions of Lemma \ref{lem:propagator} and {\Afgoneside} hold, suppose the function $\phi$ satisfies that:
\begin{enumerate}[label=(\roman*)]
    \item $\phi\in C_b^2(E)$;
    \item $\phi$ is Lipschitz continuous and
    \begin{align*}
        \|\phi\|_{\rm Lip}=\sup_{(p_1,x_1)\neq(p_2,x_2)\in E}\frac{\left|\phi(p_1,x_1)-\phi(p_2,x_2)\right|}{\left|(p_1-p_2,x_1-x_2)\right|}\leq1.
    \end{align*}
\end{enumerate}
Then for all $0\leq s\leq t\leq T$ and $(p,x)\in E$:
\begin{align}
    |\partial_x P_{s,t}\phi(p,x)|\leq{\rm e}^{Cp(t-s)},
\end{align}
where $C=C_{\overline{f},h,K,\psi}>0$ is a constant.
\end{lemma}
\begin{remark}
    Based on the proof below, condition (ii) in Lemma \ref{lem:partialx} can be weaken as $\phi$ is only Lipschitz continuous with respect to $x$ with Lipschitz coefficient no more that $1$. In other words, $|\phi(p,x_1)-\phi(p,x_2)|\leq\left|x_1-x_2\right|$ for all $p>0$ and $x_1, x_2\in R^m$.
\end{remark}
\begin{proof}
For simplicity, we use $X_r^1$ and $X_r^2$ instead of $X_r^{s,p,x_1,\mu}$ and $X_r^{s,p,x_2,\mu}$ in this proof. Note that $\|\phi\|_{\rm Lip}\leq 1$, then in the view of \eqref{propagator} and \eqref{eq:limit-sol-FPK2}, for all $x_1,x_2\in \R^m$ and $p\in\R_{+}$, it follows that
\begin{align}
    |P_{s,t}\phi(p,x_1)-P_{s,t}\phi(p,x_2)|^2=&\left|\Ex\left[\phi(p,X_t^1)-\phi(p,X_t^2)\right]\right|^2\leq\Ex\left[|X_t^1-X_t^2|^2\right].
\end{align}
An application of It\^o's formula to $|X_t^1-X_t^2|^2$ yields that
\begin{align}\label{diff:P}
    &\Ex\left[|X_t^1-X_t^2|^2\right]\nonumber\\
    =&|X_s^1-X_s^2|^2+\Ex\left[\int_s^t\left<X_r^1-X_r^2,\overline{f}\left(r,X_r^1,\left<\mu_r,\beta(X_r^1,\cdot)\right>\right)-\overline{f}\left(r,X_r^2,\left<\mu_r,\beta(X_r^2,\cdot)\right>\right)\right>dr\right]\nonumber\\
    &+\Ex\left[\int_s^t\left(\left|\left(X_r^1-X_r^2\right)+\left(h(r,X_{r-}^1)-h(r,X_{r-}^2)\right)\right|^2-|X_r^1-X_r^2|^2\right)p\psi(r,\left<\mu_r,I\right>)dr\right]\nonumber\\
    &+\Ex\left[\int_s^t\left<X_r^1-X_r^2,\overline{g}\left(r,X_r^1,\left<\mu_r,\beta(X_r^1,\cdot)\right>\right)-\overline{g}\left(r,X_r^2,\left<\mu_r,\beta(X_r^2,\cdot)\right>\right)\right>dW_r\right]\nonumber\\
    &+\Ex\left[\int_s^t\left(\left|\left(X_r^1-X_r^2\right)+\left(h(r,X_{r-}^1)-h(r,X_{r-}^2)\right)\right|^2-|X_r^1-X_r^2|^2\right)dM_r\right]\nonumber\\
    =&|x_1-x_2|^2+\Ex\left[\int_s^t\left<X_r^1-X_r^2,\overline{f}\left(r,X_r^1,\left<\mu_r,\beta(X_r^1,\cdot)\right>\right)-\overline{f}\left(r,X_r^2,\left<\mu_r,\beta(X_r^2,\cdot)\right>\right)\right>dr\right]\nonumber\\
    &+\Ex\left[\int_s^t\left(\left|\left(X_r^1-X_r^2\right)+\left(h(r,X_{r-}^1)-h(r,X_{r-}^2)\right)\right|^2-|X_r^1-X_r^2|^2\right)p\psi(r,\left<\mu_r,I\right>)dr\right].
\end{align}
By assumptions {\Afgoneside} and {\Aab}, we have the following estimate:
\begin{align}\label{diff:P1}
    &\Ex\left[\int_s^t\left<X_r^1-X_r^2,\overline{f}\left(r,X_r^1,\left<\mu_r,\beta(X_r^1,\cdot)\right>\right)-\overline{f}\left(r,X_r^2,\left<\mu_r,\beta(X_r^2,\cdot)\right>\right)\right>dr\right]\nonumber\\
    \leq&L\Ex\left[\int_s^t \left|X_r^1-X_r^2\right|^2+\left|\left<\mu_r,\beta(X_r^1,\cdot)\right>-\left<\mu_r,\beta(X_r^2,\cdot)\right>\right|^2dr\right]\nonumber\\
    \leq&L_{\overline{f}}(1+K^2)\Ex\left[\int_s^t \left|X_r^1-X_r^2\right|^2dr\right].
\end{align}
By assumption {\Ahlr} and noting that $|\psi(\cdot,\cdot)|\leq C_{\psi}$, it follows that
\begin{align}\label{diff:P2}
    &\Ex\left[\int_s^t\left(\left|\left(X_r^1-X_r^2\right)+\left(h(r,X_{r-}^1)-h(r,X_{r-}^2)\right)\right|^2-|X_r^1-X_r^2|^2\right)p\psi(r,\left<\mu_r,I\right>)dr\right]\nonumber\\
    \leq&C_{\psi}p\Ex\left[\int_s^t\left(|X_r^1-X_r^2|^2+2\left|h(r,X_{r-}^1)-h(r,X_{r-}^2)\right|^2\right)dr\right]\nonumber\\
    \leq&2C_{\psi}(1+L_h^2)p\Ex\left[\int_s^t\left|X_r^1-X_r^2\right|^2dr\right]
\end{align}
We deduce from \eqref{diff:P}, \eqref{diff:P1} and \eqref{diff:P2} that there exists a constant $C=C_{\overline{f},h,K,\psi}$ such that
\begin{align}
    \Ex\left[|X_t^1-X_t^2|^2\right]\leq |x_1-x_2|^2+Cp\int_s^t\Ex\left[|X_r^1-X_r^2|^2\right]dr.
\end{align}

Then, the Gronwall's lemma yields that
\begin{align}
    \Ex\left[|X_t^{s,p,x_1,\mu}-X_t^{s,p,x_2,\mu}|^2\right]\leq |x_1-x_2|^2{\rm e}^{Cp(t-s)}.
\end{align}
Therefore, we have
\begin{align}
    |\partial_x P_{s,t}\phi(p,x_1)|=&\lim_{x_2\rightarrow x_1}\dfrac{|P_{s,t}\phi(p,x_1)-P_{s,t}\phi(p,x_2)|}{|x_1-x_2|}\nonumber\\
    \leq&\sqrt{\dfrac{\Ex\left[|X_t^1-X_t^2|^2\right]}{|x_1-x_2|^2}}\nonumber\\
    \leq&{\rm e}^{\frac{1}{2}Cp(t-s)}.
\end{align}
This completes the proof of the lemma.
\end{proof}

\subsection{Metric and convergence of \texorpdfstring{$\mu^n$}{} to \texorpdfstring{$\mu$}{}}

Next we establish a convergence result to the solution of our FPK function. To this end, we need an appropriate metric between $\mu^n$ and $\mu$. Let ${\cal P}(E)$ be the set of all finite measures $\nu$ on ${\cal B}(E)$ such that $\nu (E)\leq 1$. To make a compact space, we add an extra point $\star$ to $E$, and denote that $E_{\star}:=E\cup\left\{\star\right\}$. Let $E$ be topologized by a topology $\Tau$, we then define a topology $\Tau^{\star}$ for $E_{\star}$ as follows: 
\begin{enumerate}[label=(\roman*)]
    \item $\Tau\subset \Tau^{\star}$;
    \item For each compact set $C\subset E$, define $U_C\in\Tau^{\star}$ by $U_C:=(E\setminus C)\cup\{\star\}$.
\end{enumerate}

We give a bijection $\iota: {\cal P}(E)\to {\cal P}(E_{\star})$ by
\begin{align}
    (\iota\nu)(A):= \nu(A\cap E)+(1-\nu(E))\delta_{\left\{\star\right\}}(A)
\end{align}
for all $A\in{\cal B}(E_{\star})$ and $\nu\in{\cal P}(E)$, where $\delta$ denotes the Dirac delta measure. Then we define the integral of $\nu\in{\cal P}(E)$ with respect to a measurable function $\phi:E_{\star}\to\R$ as:
\begin{align}\label{exinteg}
    \int_{E_{\star}}\phi(x)(\iota\nu)(dx)=\int_{E}\phi(x)\nu(dx)+\phi(\star)(1-\nu(E)).
\end{align}

For the parameter $q\geq 2$, we establish the metric $d_{q,T}$ between $\mu=(\mu_t)_{t\in[0,T]}$ and $\mu^n=(\mu^n_t)_{t\in[0,T]}$ as follows (see Ref.\cite{Eric2014Meanfield}):
\begin{align}\label{metric}
    d_{q,T}(\mu^n, \mu):=\sup_{t\in[0,T]}d_{\text{BL}}\left(\iota\mu^n_t, \iota\mu_t\right).
\end{align}
The $d_{\text{BL}}$ here is defined as
\begin{align}\label{metric_BL}
    d_{\text{BL}}\left(\iota\mu^n_t, \iota\mu_t\right)=\sup_{\phi\in{\cal R}_1}\Ex\Bigg[\left|\int_{E_{\star}}\phi(x)\left(\iota\mu^n_t-\iota\mu_t\right)(dx)\right|^q\Bigg]^{\frac{1}{q}},
\end{align}
where ${\cal R}_1$ is the set of functions $\phi:E_{\star}\to\R$ that are bounded, twice continuously differentiable and Lipschitz continuous satisfying that $\|\phi\|_{\infty}+\|\phi\|_{\rm Lip}\leq1$, where $\|\phi\|_{\infty}:=\inf\left\{C>0:\ |\phi(x)|\leq C, \ \forall x\in E_{\star}\right\}$.  

The metric $d(\cdot,\cdot)$ on $E_{\star}$ is defined as in Ref.\cite{mandelkern1989metrization}, that is: fix $x_0\in E$ and let $l(x):=\left(1+|x-x_0|\right)^{-1}$ for $x\in E$, and
\begin{align*}
d(x,y):=
\begin{cases}
    \displaystyle |x-y|\wedge|l(x)-l(y)|, &x,y\in E;\\
    \displaystyle l(x), &x\in E, y=\star;\\
    \displaystyle 0, &x=y=\star.
\end{cases}
\end{align*}
Then we have
\begin{align}
    \|\phi\|_{\rm Lip}=\sup_{x\neq y, \ x,y\in E_{\star}}\frac{|\phi(x)-\phi(y)|}{d(x,y)},
\end{align}
hence for any $\phi\in{\cal R}_1$, it holds for all $x,y\in E$ that
\begin{align}
    |\phi(x)-\phi(y)|\leq d(x,y)=|x-y|\wedge|l(x)-l(y)|\leq|x-y|.
\end{align}
Namely, the constraint of $\phi$ on $E$ is also bounded Lipschitz continuous with $\|\phi\|_{\rm Lip}\leq1$.

In view of \eqref{exinteg}, we have
\begin{align}
    &\int_{E_{\star}}\phi(x)\left(\iota\mu^n_t-\iota\mu_t\right)(dx)\nonumber\\
    =&\int_{E}\phi(x)\left(\mu^n_t-\mu_t\right)(dx)+\phi(\star)\left(\left(1-\mu^n_t(E)\right)-\left(1-\mu_t(E)\right)\right)\nonumber\\
    =&\int_{E}\phi(x)\left(\mu^n_t-\mu_t\right)(dx)\nonumber\\
    =&\left<\mu^n_t-\mu_t, \phi\right>.
\end{align}
Therefore,
\begin{align}
    d_{q,T}(\mu^n, \mu)=\sup_{t\in[0,T]}\sup_{\phi\in{\cal R}_1}\Ex\Big[\left|\left<\mu^n_t-\mu_t, \phi\right>\right|^q\Big]^{\frac{1}{q}}.
\end{align}

To obtain the convergence results, we need the following assumption:
\begin{itemize}
\item[\Aghgrow] The coefficients $\overline{g}:[0,T]\times\R^m\times\R\to\R^{md}$ and $h:[0,T]\times\R^m\to\R^m$ satisfy a growth condition. Namely, there exists a constant $K$ such that for all $t\in[0,T]$, $x\in\R^{m}$ and $y\in\R$,
\begin{align*}
    |\overline{g}(t,x,y))|^2&\leq K(1+|x|^2+|y|^2),\\
    |h(t,x)|^2&\leq K(1+|x|^2).
\end{align*}
\end{itemize}

\begin{theorem}
Let conditions of Lemma \ref{lem:partialx} and {\Aghgrow} hold, also suppose that $\Ex[|X_0^i|^q]<\infty$ for all $q\geq 2$ and $i\geq 1$. Then for fixed $T>0$ and parameters $\kappa>q$ and $m>0$, there exists a constant $C_q=C_{q,T,\overline{f},\overline{g},h,\psi,\{p_i\}_{i\geq1}, \kappa, m}>0$ which is independent of $n$, such that for all $n\geq1$:
\begin{align}
    d_{q,T}(\mu^n, \mu)\leq C_q\left(\gamma(\kappa,q,m,n)+\frac{1}{n^{1-\frac{1}{q}}}\right),
\end{align}
where $d_{q,T}(\cdot,\cdot)$ is defined by \eqref{metric}, and
\begin{align}\label{def:gamma}
    \gamma(\kappa, q, m,n):=\left\{
    \begin{array}{lcr}
        n^{-\frac{1}{2}}+n^{-\frac{\kappa-q}{\kappa}}, &q>\frac{m}{2}, \kappa\neq 2q;\\
        n^{-\frac{1}{2}}\ln(1+n)+n^{-\frac{\kappa-q}{\kappa}}, &q=\frac{m}{2}, \kappa\neq 2q;\\
        n^{-\frac{q}{m}}+n^{-\frac{\kappa-q}{\kappa}}, &q<\frac{m}{2}, \kappa\neq\frac{m}{m-q}.
    \end{array}
    \right.
\end{align}
\end{theorem}
    
\begin{proof}
By Lemma \ref{lem:weakconverge}, it holds for all $\phi\in {\cal R}_1$ and $0\leq t\leq T$ that
\begin{align}
    &\left<\mu_t^n-\mu_t,\phi\right>\nonumber\\
    =&\left<\mu_0^n-\mu_0,P_{0,t}\phi\right>+\frac{1}{n}\sum_{i=1}^{n}\int_0^t\partial_x P_{s,t}\phi(p_i,X_s^i)\overline{g}(s,X_t^i,\frac{1}{n}\sum_{j=1}^{n}\beta(X_s^i,X_s^j))dW_s^i\nonumber\\
    &+\frac{1}{n}\sum_{i=1}^{n}\int_0^t (P_{s,t}\phi(p_i,X_{s-}^i+h(s,X_{s-}^i))-P_{s,t}\phi(p_i,X_{s-}^i))dM_s^i\nonumber\\
    =:&I_1(t)+I_2(t)+I_3(t).
\end{align}
Then we have that
\begin{align}
    &\Ex\left[\left|\left<\mu_t^n-\mu_t,\phi\right>\right|^q\right]
    =\Ex\left[\left|I_1(t)+I_2(t)+I_3(t)\right|^q\right]\nonumber\\
    \leq &C_q\left(\Ex\left[\left|I_1(t)\right|^q\right]+\Ex\left[\left|I_2(t)\right|^q\right]+\Ex\left[\left|I_3(t)\right|^q\right]\right).
\end{align}
In what follows, we will give the estimate of $\Ex\left[\left|I_1(t)\right|^q\right]$, $\Ex\left[\left|I_2(t)\right|^q\right]$ and $\Ex\left[\left|I_3(t)\right|^q\right]$. One is reminded that during the following proof, the constant $C_q$ can be different from line to line. 

For the first term, by using Theorem 2 in Ref.\cite{bo2022probabilistic}, we have
\begin{align}\label{i1}
    \Ex\left[\left|I_1(t)\right|^q\right]\leq C_{q,\kappa,m}\gamma(\kappa,q,m,n),
\end{align}
where $\gamma(\kappa,q,m,n)$ is defined as in \eqref{def:gamma}. (For this estimate, Ref.\cite{fournier2015rate} can also be referenced.) 

For the second term, noting that $\phi\in {\cal R}_1$, by Burkholder-Davis-Gundy inequality, Lemma \ref{lem:partialx} and assumption {\Aghgrow}, it holds that
\begin{align}\label{eq:I2_1}
    &\Ex\left[\left|I_2(t)\right|^q\right]=\Ex\left[\left|\frac{1}{n}\sum_{i=1}^{n}\int_0^t\partial_x P_{s,t}\phi(p_i,X_s^i)\overline{g}\left(s,X_s^i,\frac{1}{n}\sum_{j=1}^{n}\beta(X_s^i,X_s^j)\right)dW_s^i\right|^q\right]\nonumber\\
    \leq&\Ex\left[\sup_{t'\in[0,t]}\left|\frac{1}{n}\sum_{i=1}^{n}\int_0^{t'}\partial_x P_{s,t'}\phi(p_i,X_s^i)\overline{g}\left(s,X_s^i,\frac{1}{n}\sum_{j=1}^{n}\beta(X_s^i,X_s^j)\right)dW_s^i\right|^q\right]\nonumber\\
    \leq&\Ex\left[\Bigg(\frac{1}{n}\sum_{i=1}^{n}\int_0^t\left|\partial_x P_{s,t}\phi(p_i,X_s^i)\right|^2\left|\overline{g}\left(s,X_s^i,\frac{1}{n}\sum_{j=1}^{n}\beta(X_s^i,X_s^j)\right)\right|^2ds\Bigg)^{\frac{q}{2}}\right]\nonumber\\
    \leq&\Ex\left[\left(\frac{1}{n}\sum_{i=1}^{n}\int_0^t{\rm e}^{Cp_i t}\left|\overline{g}\left(s,X_s^i,\frac{1}{n}\sum_{j=1}^{n}\beta(X_s^i,X_s^j)\right)\right|^2ds\right)^{\frac{q}{2}}\right]\nonumber\\
    \leq&\frac{C_q}{n^q}\Ex\left[\left(\sum_{i=1}^{n}\int_0^t\left|\overline{g}\left(s,X_s^i,\frac{1}{n}\sum_{j=1}^{n}\beta(X_s^i,X_s^j)\right)\right|^2ds\right)^{\frac{q}{2}}\right]\nonumber\nonumber\\
    \leq&\frac{C_q}{n^q}\sum_{i=1}^{n}\Ex\Bigg[\Bigg(\int_0^t\left|\overline{g}\left(s,X_s^i,\frac{1}{n}\sum_{j=1}^{n}\beta(X_s^i,X_s^j)\right)\right|^2ds\Bigg)^{\frac{q}{2}}\Bigg]\nonumber\\
    \leq&\frac{C_q}{n^q}\sum_{i=1}^{n}\Ex\Bigg[\Bigg(\int_0^t\left(1+|X_s^i|^2+\frac{1}{n}\sum_{j=1}^{n}|X_s^j|^2\right)ds\Bigg)^{\frac{q}{2}}\Bigg].
\end{align}
Here we have used the convergence of $(p_i)_{i\geq 1}$, which leads to the existence of a constant $C'>0$ such that ${\rm e}^{Cp_i}<C'$ for all $i\geq 1$. In addition, since $\frac{q}{2}\geq1$, applying Jensen's inequality twice to the convex function $a(x)=x^{\frac{q}{2}}$ leads to
\begin{align}
    &\left(\int_0^t\left(1+|X_s^i|^2+\frac{1}{n}\sum_{j=1}^{n}|X_s^j|^2\right)ds\right)^{\frac{q}{2}}\nonumber\\
    =&3^{\frac{q}{2}}\left(\frac{t+\int_0^t|X_s^i|^2ds+\left(\frac{1}{n}\sum_{j=1}^{n}\int_0^t|X_s^j|^2ds\right)}{3}\right)^{\frac{q}{2}}\nonumber\\
    \leq&3^{\frac{q}{2}-1}\Bigg(t^{\frac{q}{2}}+\left(\int_0^t|X_s^i|^2ds\right)^{\frac{q}{2}}+\left(\frac{\sum_{j=1}^{n}\int_0^t|X_s^j|^2ds}{n}\right)^{\frac{q}{2}}\Bigg)\nonumber\\
    \leq&3^{\frac{q}{2}-1}\Bigg(t^{\frac{q}{2}}+\left(\int_0^t|X_s^i|^2ds\right)^{\frac{q}{2}}+\frac{1}{n}\sum_{j=1}^{n}\left(\int_0^t|X_s^j|^2ds\right)^{\frac{q}{2}}\Bigg).
\end{align}
Summing over $i=1,2,\ldots,n$, multiplying by $\frac{C_q}{n^q}$ and taking expectation, it follows that
\begin{align}\label{eq:I2_2}
    &\frac{C_q}{n^q}\sum_{i=1}^{n}\Ex\Bigg[\Bigg(\int_0^t\left(1+|X_s^i|^2+\frac{1}{n}\sum_{j=1}^{n}|X_s^j|^2\right)ds\Bigg)^{\frac{q}{2}}\Bigg]\nonumber\\
    \leq&\frac{C_q}{n^q}\sum_{i=1}^{n}\Ex\Bigg[3^{\frac{q}{2}-1}\Bigg(t^{\frac{q}{2}}+\left(\int_0^t|X_s^i|^2ds\right)^{\frac{q}{2}}+\frac{1}{n}\sum_{j=1}^{n}\left(\int_0^t|X_s^j|^2ds\right)^{\frac{q}{2}}\Bigg)\Bigg]\nonumber\\
    \leq&\frac{C_q}{n^q}\sum_{i=1}^{n}\Ex\Bigg[t^{\frac{q}{2}}+2\left(\int_0^t|X_s^i|^2ds\right)^{\frac{q}{2}}\Bigg]\nonumber\\
    \leq&\frac{C_q}{n^q}\sum_{i=1}^{n}\Ex\Bigg[1+\left(\int_0^t|X_s^i|^2ds\right)^{\frac{q}{2}}\Bigg].
\end{align}
Hence in view of \eqref{eq:I2_1} and \eqref{eq:I2_2}, we have
\begin{align}\label{eq:I2_3}
    &\Ex\left[\left|I_2(t)\right|^q\right]\leq\frac{C_q}{n^q}\sum_{i=1}^{n}\Ex\Bigg[\Bigg(\int_0^t\left(1+|X_s^i|^2+\frac{1}{n}\sum_{j=1}^{n}|X_s^j|^2\right)ds\Bigg)^{\frac{q}{2}}\Bigg]\nonumber\\
    \leq&\frac{C_q}{n^q}\sum_{i=1}^{n}\Ex\Bigg[1+\left(\int_0^t|X_s^i|^2ds\right)^{\frac{q}{2}}\Bigg]\nonumber\\
    =&\frac{C_q}{n^{q-1}}\left(1+\frac{1}{n}\sum_{i=1}^{n}\Ex\Bigg[\left(\int_0^t|X_s^i|^2ds\right)^{\frac{q}{2}}\Bigg]\right).
\end{align}
By Minkowski's integral inequality\textsuperscript{\cite{hardy1952inequalities}} for $\frac{q}{2}\geq1$, we know that
\begin{align}\label{eq:I2_4}
    \Ex\Bigg[\left(\int_0^t|X_s^i|^2ds\right)^{\frac{q}{2}}\Bigg]\leq\Bigg(\int_0^t\left(\Ex\left[|X_s^i|^q\right]\right)^{\frac{2}{q}}ds\Bigg)^{\frac{q}{2}}.
\end{align}
Then through an application of Gronwall's lemma together with assumption {\Ai}, one can verify that there exists a constant $C_{T,q}>0$ such that
\begin{align}\label{eq:I2_5}
    \Ex\left[|X_s^i|^q\right]\leq C_{T,q}\left(1+\Ex\left[|X_0^i|^q\right]\right)=C_{T,q}\left(1+\Ex\left[|X_0^1|^q\right]\right),
\end{align}
see Theorem 67 of Ref.\cite{philip2013stochastic} for the details. Therefore, \eqref{eq:I2_3}, \eqref{eq:I2_4} and \eqref{eq:I2_5} yield that
\begin{align}\label{i2}
    \Ex\left[\left|I_2(t)\right|^q\right]\leq\frac{C_q}{n^{q-1}}\left(1+\frac{1}{n}\sum_{i=1}^{n}\Ex\Bigg[\left(\int_0^t|X_s^i|^2ds\right)^{\frac{q}{2}}\Bigg]\right)\leq\frac{C_q}{n^{q-1}}\left(1+\Ex\left[|X_0^1|^q\right]\right).
\end{align}

For the third term, using the mean value theorem to $P_{s,t}\phi(p_i,\cdot)$ with $\xi_{s-}^i\in[0, h(t,X_{s-}^i)]$ for all $s\in(0,t]$, with the help of Lemma \ref{lem:propagator}, Lemma \ref{lem:partialx}, assumption {\Aghgrow}, the boundedness of $(p_i)_{i\geq1}$ and $\psi$, and by some analogous methods as in the previous one, it holds that
\begin{align}\label{i3}
    & \Ex\left[\left|I_3(t)\right|^q\right]=\Ex\left[\left|\frac{1}{n}\sum_{i=1}^{n}\int_0^t P_{s,t}\phi(p_i,X_{s-}^i+h(s,X_{s-}^i))-P_{s,t}\phi(p_i,X_{s-}^i)dM_s^i\right|^q\right]\nonumber\\
    =&\Ex\left[\left|\frac{1}{n}\sum_{i=1}^{n}\int_0^t \partial_x P_{s,t}\phi(p_i,X_{s-}^i+\xi_{s-}^i)h(t,X_{s-}^i)dM_s^i\right|^q\right]\nonumber\\
    \leq&\Ex\left[\sup_{t'\in[0,t]}\left|\frac{1}{n}\sum_{i=1}^{n}\int_0^{t'}\partial_x P_{s,t'}\phi(p_i,X_{s-}^i+\xi_{s-}^i)h(s,X_{s-}^i)dM_s^i\right|^q\right]\nonumber\\
    \leq&\Ex\left[\left|\frac{1}{n}\sum_{i=1}^{n}\int_0^{t} \left(\partial_x P_{s,t}\phi(p_i,X_{s-}^i+\xi_{s-}^i)\right)^2h(s,X_{s-}^i)^2d\left<M^i\right>_s\right|^{\frac{q}{2}}\right]\nonumber\\
    \leq&\Ex\left[\left(\frac{1}{n}\sum_{i=1}^{n}\int_0^{t}\left|\partial_x P_{s,t}\phi(p_i,X_{s-}^i+\xi_{s-}^i)\right|^2\left|h(s,X_{s-}^i)\right|^2p_i \psi\left(s,\frac{1}{n}\sum_{j=1}^{n}X_{s}^j\right)ds\right)^{\frac{q}{2}}\right]\nonumber\\
\leq&C_q\Ex\left[\left(\frac{1}{n}\sum_{i=1}^{n}\int_0^{t}\left|h(s,X_{s-}^i)\right|^2ds\right)^{\frac{q}{2}}\right]\nonumber\\
    \leq&\frac{C_q}{n^q}\sum_{i=1}^{n}\Ex\Bigg[\Bigg(\int_0^t\left|h(s,X_{s-}^i)\right|^2 ds \Bigg)^{\frac{q}{2}}\Bigg]\nonumber\\
    \leq&\frac{C_q}{n^q}\sum_{i=1}^{n}\Ex\Bigg[\Bigg(\int_0^t\left(1+|X_{s-}^i|^2\right)ds \Bigg)^{\frac{q}{2}}\Bigg]\nonumber\\
    \leq&\frac{C_q}{n^q}\sum_{i=1}^{n}\Ex\Bigg[ 1+\left(\int_0^t|X_{s}^i|^2ds\right) ^{\frac{q}{2}}\Bigg]\nonumber\\
    =&\frac{C_q}{n^{q-1}}\left(1+\frac{1}{n}\sum_{i=1}^{n}\Ex\Bigg[\left(\int_0^t|X_s^i|^2ds\right)^{\frac{q}{2}}\Bigg]\right)\nonumber\\
    \leq&\frac{C_q}{n^{q-1}}\left(1+\Ex\left[\left|X_0^1\right|^q\right]\right).
\end{align}
Note that $\Ex\left[\left|X_0^1\right|^q\right]<\infty$, then in view of \eqref{metric}, \eqref{metric_BL}, \eqref{i1}, \eqref{i2} and \eqref{i3}, we complete the proof.
\end{proof}

\section*{Conflict of Interest}

The author declares that he has no conflict of interest.

\section*{Biographies}

Zeqian Li is currently a postgraduate student under the tutelage of Prof. Lijun Bo at the School of Mathematics, University of Science and Technology of China. His research interests focus on mean-field analysis.


\nocite{*}
\bibliographystyle{custom-unsrt}
\bibliography{reference.bib}

\end{document}